\documentclass{amsart}
\usepackage[utf8]{inputenc}
\usepackage{fullpage}
\usepackage{amssymb}
\usepackage{hyperref}
\hypersetup{colorlinks=true, linkcolor=blue, filecolor=blue, urlcolor=blue, citecolor=blue}
\usepackage{enumerate}

\newcommand{\NN}{\mathbb{N}}
\newcommand{\ZZ}{\mathbb{Z}}
\newcommand{\QQ}{\mathbb{Q}}
\newcommand{\RR}{\mathbb{R}}
\newcommand{\CC}{\mathbb{C}}

\newcommand{\cD}{\mathcal{D}}
\newcommand{\cP}{\mathcal{P}}

\newcommand{\PGL}{\mathrm{PGL}}
\newcommand{\PU}{\mathrm{PU}}
\newcommand{\PO}{\mathrm{PO}}
\newcommand{\SU}{\mathrm{SU}}
\newcommand{\SL}{\mathrm{SL}}
\newcommand{\GL}{\mathrm{GL}}

\newcommand{\bs}{\backslash}
\newcommand{\ma}{\mathfrak{a}}
\newcommand{\eps}{\varepsilon}

\newcommand{\rk}{\operatorname{rk}}
\newcommand{\diag}{\operatorname{diag}}
\newcommand{\id}{\operatorname{id}}
\newcommand{\den}{\operatorname{den}}
\newcommand{\opint}{\operatorname{int}}
\newcommand{\dist}{\operatorname{dist}}
\newcommand{\comp}{\operatorname{comp}}
\newcommand{\Hom}{\operatorname{Hom}}
\newcommand{\proj}{\operatorname{proj}}

\renewcommand{\geq}{\geqslant}
\renewcommand{\leq}{\leqslant}

\usepackage[dvipsnames]{xcolor}

\newtheorem{theorem}{Theorem}
\newtheorem{proposition}{Proposition}
\newtheorem{lemma}{Lemma}

\title{The sup-norm problem for automorphic cusp forms of $\PGL(n,\ZZ[i])$}
\author{Péter Maga}
\author{Gergely Zábrádi}
\date{}

\address{Alfréd Rényi Institute of Mathematics, Hungarian Academy of Sciences, POB 127, Budapest H-1364, Hungary}
\email{magapeter@gmail.com, gergely.zabradi@ttk.elte.hu}
\address{Eötvös Loránd University, Institute of Mathematics, Pázmány Péter sétány 1/C, Budapest H-1117, Hungary}
\email{gergely.zabradi@ttk.elte.hu}

\begin{document}



\maketitle

\begin{abstract}
Let $\phi$ be an $L^2$-normalized Hecke--Maa{\ss} cusp form for $\PGL_n(\ZZ[i])$ on the locally symmetric space $X:=\PGL_n(\ZZ[i])\bs \PGL_n(\CC) / \PU_n$. If $\Omega$ is a compact subset of $X$, then we prove the bound $\|\phi|_{\Omega}\|_{\infty}\ll_{\Omega} \lambda_{\phi}^{n(n-1)/4-\delta}$ for some $\delta>0$ depending only on $n$, where $\lambda_{\phi}$ is the Laplace eigenvalue of $\phi$.
\end{abstract}

\section{Introduction}

Given a Riemannian locally symmetric space $X=\Gamma \bs S$, it is a classical analytic problem to give pointwise bounds for ($L^2$-normalized) eigenfunctions $F$ of the algebra of invariant differential operators $\cD(S)$, uniformly in $S$ in  terms of the Laplace eigenvalue $\lambda_F$ of $F$. If $X$ is compact, then Sarnak \cite{Sarnak2004L} proved the baseline bound
\begin{equation}\label{eq:sarnak-baseline}
\|F\|_{\infty} \ll_X \lambda_F^{\frac{\dim(X)-\rk(X)}{4}}.
\end{equation}
The exponent on the right-hand side is known to be sharp in general. It is also known in some special cases that if $X$ is not compact, then $\|F\|_{\infty}$ might be considerably larger (see \cite{BrumleyTemplier2020}), but \eqref{eq:sarnak-baseline} still holds if $F$ is retsriced to compact subsets of $X$. The sup-norm problem in the theory of automorphic forms asks if the exponent of \eqref{eq:sarnak-baseline} can be strengthened if $X$ is an arithmetic manifold and $F$ is an eigenfunction not only of $\cD(S)$ but of the full Hecke algebra of $X$.

An important motivation comes from quantum mechanics. Classical mechanics interprets a freely moving particle as a geodesic flow. The quantum mechanical interpretation of the same object is an $L^2$-normalized linear combination of eigenstates. Since the geodesic flow is ergodic with respect to the Liouville measure on the tangent space, the correspondence principle of quantum mechanics suggests that the masses of the eigenstates reproduce the invariant measure in the high-energy limit. Bounds on the sup-norm (or in general, the $L^p$ norm for any $p>2$) of eigenstates control their mass concentration, and hence are in connection with the quantum unique ergodicity conjecture of Rudnick and Sarnak \cite{RudnickSarnak1994}. Other important connections of sup-norm bounds are towards the multiplicity problem, nodal domains of automorphic forms and bounds for $L$-functions, see e.g. \cite{Sarnak2004L}, \cite{GhoshReznikovSarnak}, \cite{BlomerHarcos2010}.

In improving \eqref{eq:sarnak-baseline}, there are (at least) two independent directions of research: one is to find as strong power-savings as possible among special circumstances, typically in rank one (see e.g. \cite{IwaniecSarnak1995} and \cite{BlomerHarcosMagaMilicevic2020}), the other one is to find any power-saving in higher rank or among as general circumstances as possible (see e.g. \cite{BlomerPohl2016}, \cite{BlomerMaga2015}, \cite{BlomerMaga2016}). We pick up the thread in the second theme, where the current limitation of our knowledge is an unpublished a manuscipt \cite{Marshall2014+}, which proves power-saving for a wide class of symmetric spaces, more specifically, for arithmetic quotients of quasi-split groups with the exception of the type $\SU(n,n-1)$. (We note that even though \cite{Marshall2014+} is not peer-reviewed, it is widely accepted in the community as being correct.)

In these notes, we introduce a new method, based and improved on that of \cite{BlomerMaga2016}, to tackle the sup-norm problem. As a test application, we prove a saving over \eqref{eq:sarnak-baseline} in the case of the locally symmetric space
\[
X:=\Gamma \bs G / K,\qquad \Gamma:=\PGL_n(\ZZ[i]),\qquad G:=\PGL_n(\CC), \qquad K:=\PU_n.
\]
Admittedly, our result follows from the main result of \cite{Marshall2014+} (see especially \cite[Corollary~1.4]{Marshall2014+}), but we strongly believe that the novelty of the method deserves attention, and might have the potential to address the type $\SU(n,n-1)$, the exceptional case in Marshall's work.

All along the paper, we think of $n\geq 2$ as being fixed, in particular, all implied constants below are allowed to depend on $n$.

Before stating our main result, we fix some notations. Since it is convenient to work with matrices instead of their projectivization, we shall realize elements of $G$ as the rightmost nonzero entry of the bottom row is $1$. We will also talk about points of $X$ or $G/K$ as matrices, by which we mean any matrix which represents them. Let $A$ stand for the diagonal subgroup of $G$ consisting of positive real entries and let $N$ be the upper-triangular unipotent subgroup. Then the Iwasawa decomposition reads as $G=NAK$. We write $\ma$ for the Lie algebra of $A$, $\ma^*$ for its dual and $\ma^*_{\CC}$ for the complexification of $\ma^*$. Fixing bases in $\ma^*$ and $\ma^*_{\CC}$, we may view them as appropriate subsets of $\RR^n$ and $\CC^n$, respectively. Let $W$ stand for the Weyl group, $\Sigma$ for the set of roots and $\Sigma^+$ for the set of positive roots corresponding to $N$. For $\alpha\in\Sigma$, let $m(\alpha)$ be the real dimension of the corresponding root space. For $\lambda\in\ma^*$, define
\[
D(\lambda):=\prod_{\alpha\in\Sigma^+} (1+|\langle \alpha,\lambda \rangle|)^{m(\alpha)},
\]
where $\langle \cdot, \cdot \rangle$ is induced by the Killing form.

In our special situation, define, for $1\leq j\leq n$, the function $e_j$ on $\ma$ as $e_j(\diag(a_1,\dots,a_n))=a_j$, and then a set of positive roots is given by $e_j-e_k$ with $1\leq j<k\leq n$. The corresponding root space is spanned by the matrices $E_{jk}$ and $iE_{jk}$, where $E_{jk}$ is the matrix which $1$ at position $j,k$ and otherwise zero, hence $m(e_j-e_k)=2$.

A Hecke--Maa{\ss} cusp form $\phi$ on $X$ comes with archimedean Langlands parameters $\mu_{\phi}=(\mu_1,\dotsc,\mu_n)\in \ma^*_{\CC}/W\subset \CC^n/W$ such that $\mu_1+\dotsc+\mu_n=0$ and $\{\mu_1,\dotsc,\mu_n\}=\{\overline{\mu_1},\dotsc,\overline{\mu_n}\}$.
In our parametrization, $(\mu_1,\dotsc,\mu_n)\in\RR^n$ corresponds to the tempered spectrum, and we have that $\Im \mu_1,\dotsc, \Im \mu_n = O(1)$. Let $\mu_{\phi}^*$ stand for $\Re\mu_{\phi}$. Since for the Laplace eigenvalue $\lambda_{\phi}$ of $\phi$,
\[
\lambda_{\phi} \asymp 1 + \|\mu_1\|^2 + \dotsc + \|\mu_n\|^2
\]
holds, we have
\[
D(\mu_{\phi}^*) \ll \lambda_{\phi}^{\frac{n(n-1)}{2}}.
\]

The main result of this paper is the following.
\begin{theorem}\label{thm:main} For any $n\geq 2$, there exists some $\delta=\delta(n)>0$ with the following property. For any $L^2$-normalized Hecke--Maa{\ss} cusp form $\phi$ on $X$ and any compact $\Omega\subset X$,
\[
\|\phi|_{\Omega}\|_{\infty} \ll_{\Omega} D(\mu_{\phi}^*)^{\frac{1}{2}-\delta}.
\]
In particular,
\[
\|\phi|_{\Omega}\|_{\infty} \ll_{\Omega} \lambda_{\phi}^{\frac{n(n-1)}{4}(1-2\delta)}.
\]
\end{theorem}
Since $n(n-1)/4=(\dim(X)-\rk(X))/4$, this is a saving over \eqref{eq:sarnak-baseline}. We note that the case $n=2$ was solved in \cite{BlomerHarcosMilicevic2016}.

Our method closely follows the one introduced in \cite{BlomerMaga2016}, however, the counting problem in this situation is more challenging (because of the too large maximal compact subgroup of $G$, as it will be briefly exposed below). This requires an improvement in the counting techniques, which is the heart of this paper (vaguely speaking, a combination of Lemma~\ref{lemma:spec-lemma-two-primes} and Lemma~\ref{lemma:constructQ3}). As a by-product, there is no need of Linnik type theorems about small primes in arithmetic progressions (i.e. zero repulsion of $L$-functions) any more.

We remark that with some work, the jungle of $O(1)$'s below (particularly in Section~\ref{sec:exchanging}) can be made explicit, hence an explicit subconvexity saving is available, see the work \cite{Gillman} over the rational field (in fact, now it is easier, since the reference to Linnik type theorems is removed). We also note that the implied constant (which depends on $\Omega$) can be also made fully explicit, so the method in principle is effective, e.g. we do not need any reference to Siegel zeroes.

\subsection*{Acknowledgements} We thank Valentin Blomer for useful discussions about the topic of the paper. We also thank Gergely Harcos and Vitezslav Kala for discussions about some closely related questions.

The research towards this work was supported by NKFIH (National Research, Development and Innovation Office) Grants KKP~133819 (PM), FK~135218 (PM), FK~127906 (GZ), K~135885 (GZ), ELKH (Eötvös Loránd Research Network) Grant SA-71/2021 (PM \& GZ), and the MTA R\'enyi Int\'ezet Lend\"ulet Automorphic Research Group (PM \& GZ).

\section{Reduction to a counting problem}

In this section, we reduce the problem to a matrix counting problem, following the lines of \cite{BlomerMaga2015} and \cite{BlomerMaga2016}, where $\PGL_n(\ZZ) \bs \PGL_n(\RR) / \PO_n$ was treated. Since the spherical Hecke algebra is isomorphic in that case to ours at the archimedean and all relevant non-archimedean places, almost everything follows verbatim, so we give only a brief exposition, with some emphasis on the small difference coming from the fact that the maximal compact subgroup is the orthogonal group in the real case and the unitary group in the complex case.

Let $f_{\mu_{\phi}}:K\bs G / K \to\CC$ be the spherical function constructed in \cite[pp.1276--1277]{BlomerMaga2016} (here we utilize that $A,\ma^*,\ma^*_{\CC}$ are the same for $\PGL_n(\RR)$ and $\PGL_n(\CC)$). Denoting by $\tilde{f}_{\mu_{\phi}}$ its spherical transform, for any $x,y\in G$, we have the pre-trace formula
\begin{equation}\label{eq:pre-trace-formula}
\int \tilde{f}_{\mu_{\phi}}(\mu_\varpi)  \overline{F_{\varpi}(x)} F_{\varpi}(y)\ d\varpi = \sum_{\gamma\in \Gamma} f_{\mu_{\phi}}(x^{-1}\gamma y),
\end{equation}
where the integral on the left is meant over the full spectrum of $\Gamma \bs G$.

The idea of amplification in the current setup can be summarized as follows. Assume we want to estimate our form $\phi$ at a point $g\in \Omega$. Then in \eqref{eq:pre-trace-formula}, we set $x:=g$ and we take $y$ of the form $\gamma g$, where $\gamma$ runs through right coset representatives corresponding to certain Hecke operators. An appropriate weighted sum of these results in a positive operator on $L^2(X)$, hence from the left-hand side, all but the one term corresponding to $F_{\varpi}=\phi$ can be dropped (for this positivity argument, see also \cite[Section~3]{BlomerHarcosMagaMilicevic2020}). All in all, we arrive at a bound of the form
\[
C |\phi(g)|^2 \leq \sum_{\gamma\in \Gamma'} f_{\mu_{\phi}}(g^{-1}\gamma g)
\]
with $\Gamma'$ being a finite set of matrices arising from the Hecke operators utilized, and $C$ is a positive constant coming from $\tilde{f}_{\mu_{\phi}}$ and the corresponding Hecke eigenvalues. Then, using that $f_{\mu_{\phi}}$ decays controllably away from $K$ (see \cite[Theorem~2]{BlomerPohl2016}), to bound $|\phi(g)|$, we essentially have to count those $\gamma\in\Gamma'$ for which $g^{-1}\gamma g$ is close to $K$.

To be a little more concrete, let $L\geq 2$ be a parameter, and assume that $\cP$ is a set of primes $\pi$ lying above distinct split rational primes such that $L\leq N(\pi)<2L$ with $N$ standing for the norm. Then the discussion of \cite[Sections~4,~6]{BlomerMaga2015} (which in fact uses \cite[Theorem~2]{BlomerPohl2016}, and see also \cite[Section~2]{BlomerMaga2016}) leads to, for any number $M\geq 1$,
\begin{equation}\label{eq:key-inequality}
|\phi(g)|^2 \ll_{M,\Omega} D(\mu_{\phi}^*) \cdot\left(\frac{1}{\#\cP} + D(\mu_{\phi}^*)^{-\kappa} L^{K} + \sum_{\nu=1}^n \frac{1}{(\#\cP)^2} \sum_{\pi,\pi'\in\cP} \frac{\#S(Q,\pi^{\nu},\pi'^{\nu},M)}{L^{\nu(n-1)}}\right),\qquad g\in\Omega,
\end{equation}
where $K$ is a fixed number depending only on $M$ and $n$, $\kappa>0$ is also a fixed number depending only on $n$; and
\begin{equation}\label{eq:def-of-S}
\begin{split}
S(Q,\pi^{\nu},\pi'^{\nu},M): =  \Bigl\{ & \gamma\in \SL_n(\ZZ[i])\cdot\diag(1,\pi^{\nu},\dotsc,\pi^{\nu},\pi^{\nu}\pi'^{\nu})\cdot\SL_n(\ZZ[i]): \\ & \||\det\gamma|^{-\frac{2}{n}}\cdot \gamma^* Q \gamma -  Q \|_{\infty} \leq \max\left(N(\pi)^{-M},N(\pi')^{-M}\right)\Bigr\}
\end{split}
\end{equation}
with $Q:=|\det g|^{2/n}\cdot (g^*)^{-1} g^{-1}$, and by $\| \cdot \|_{\infty}$ applied to a matrix, we mean its largest entry in absolute value. Note that in \cite{BlomerMaga2015} and \cite{BlomerMaga2016}, the transpose of $\gamma$ and $g$ are taken instead of their adjoint.

By the prime number theorem, we can choose $\cP$ to satisfy $\#\cP\gg_{\eps}L^{1-\eps}$ for any $\eps>0$. Therefore, if we were able to prove that
\[
\#S(Q,\pi^{\nu},\pi'^{\nu},M) \ll_{\Omega,M} L^{\nu(n-1)-\eta}
\]
holds for every $1\leq \nu\leq n$ with some $\eta>0$, then \eqref{eq:key-inequality} would imply Theorem~\ref{thm:main} by setting $L$ to be a very small positive power of $D(\mu_{\phi}^*)$. We prove it in a weaker form, which still suffices for Theorem~\ref{thm:main}.

Before formulating this weaker statement, we fix some notation. Introduce the notation $S_n$ for the real vector space of self-adjoint matrices in $\CC^{n\times n}$ and $P_n\subset S_n$ for the open convex cone of positive definite matrices. Given a compact set $\Omega\subset X$ as in Theorem~\ref{thm:main}, we introduce
\[
\Omega':= \left\{\left[\frac{1}{2},2\right]\cdot|\det g|^{\frac{2}{n}}\cdot (g^*)^{-1} g^{-1}:g\in\Omega\right\}.
\]
Then $\Omega'$ is a compact subset of $P_n$ in the subspace topology.

\begin{proposition}\label{prop:main} Let $\eps>0$ be arbitrary. There exist positive numbers $\alpha(\eps)\geq 1$ and $M(\eps)\geq 1$ (both depending only on $n$ and $\eps$) with the following properties. For any compact $\Omega\subset X$, there exists a constant $L(\Omega)\geq 2$ (depending only on $n$, $\alpha(\eps)$, $M(\eps)$ and $\Omega$) such that the following holds. For any $L_0\geq L(\Omega)$, there exists some $L_0 \leq L \leq L_0^{\alpha(\eps)}$ such that for any prime $\pi$ lying above a split rational prime and satisfying $L\leq N(\pi)<2L$, we have
\begin{equation}\label{eq:prop-counting-one-prime}
\#S(Q,\pi^{\nu},\pi^{\nu},M(\eps)) \ll_{\Omega,\eps} L^{\nu(n-1)+\eps},\qquad Q\in\Omega', 1\leq \nu\leq n;
\end{equation}
and for any two distinct primes $\pi,\pi'$ lying above distinct split rational primes, we have
\begin{equation}\label{eq:prop-counting-two-primes}
\#S(Q,\pi^{\nu},\pi'^{\nu},M(\eps)) = 0,\qquad Q\in\Omega', 1\leq \nu\leq n.
\end{equation}
\end{proposition}
This still suffices for the proof of Theorem~\ref{thm:main}. Indeed, we apply Proposition~\ref{prop:main} with any $0<\eps<1/2$. There are implied numbers $\alpha:=\alpha(\eps)\geq 1$ and $M:=M(\eps)\geq 1$, and together with the further input $\Omega$, one more number $L(\Omega)$ by Proposition~\ref{prop:main}. Then take $L_0:=D(\mu_{\phi}^*)^{\omega}$, where $\omega>0$ is a fixed constant to be specified later. By Proposition~\ref{prop:main}, if $L_0\geq L(\Omega)$, we can find some with $D(\mu_{\phi}^*)^{\omega}\leq L\leq D(\mu_{\phi}^*)^{\alpha\omega}$ such that the countings \eqref{eq:prop-counting-one-prime} and \eqref{eq:prop-counting-two-primes} hold for any primes $L\leq N(\pi),N(\pi')<2L$. Then in every term of \eqref{eq:key-inequality}, we get a power-saving, as soon as $\omega>0$ in the beginning is chosen sufficiently small. Now we return to the condition $L_0\geq L(\Omega)$. Apart from a finite set of Hecke--Maa{\ss} cusp forms depending only on $\Omega$, this is indeed satisfied by $L_0=D(\mu_{\phi}^*)^{\omega}$. The finitely many exceptional forms are treated then by adjusting the implied constant (note that the implied constant might depend on $\Omega$, only the saving in the exponent must be absolute).

The rest of the paper is hence devoted to the proof of Proposition~\ref{prop:main}.

We conclude this section by illustrating why this counting problem is harder for $\PGL_n(\ZZ[i])$ than for $\PGL_n(\ZZ)$ in the special case when $Q$ is the unit matrix. Then in $S(\dotsc)$, we have a condition on the Smith normal form which is of the same complexity in both cases. However, the other condition is that $\gamma$ is projectively equivalent to an orthonormal matrix, which intuitively happens more often over $\CC$ than over $\RR$, since
\[
\dim(\PGL_n(\CC))=2n^2-2,\qquad \dim(\PU_n)=n^2-1=\frac{\dim(\PGL_n(\CC))}{2},
\]
while
\[
\dim(\PGL_n(\RR))=n^2-1, \qquad \dim(\PO_n)=\frac{n^2-n}{2}
<\frac{\dim(\PGL_n(\RR))}{2},
\]
i.e. $\PU_n$ is a ``thicker'' subgroup of $\PGL_n(\CC)$ than $\PO_n$ of $\PGL_n(\RR)$.

\section{Counting techniques in a special case}

In this section, we introduce counting techniques for a situation which is very special in many different aspects. First, we will assume that $Q$ is diagonal, moreover, its entries belong to the base field $\QQ(i)$. Then these together imply that the diagonal entries are rational, i.e. $Q=\diag(q_1,\dotsc,q_n)$ with $q_1,\dotsc,q_n\in\QQ$. Then $q_j\asymp_{\Omega} 1$ for all $1\leq j\leq n$, and we also assume that the numerator and the denominator of $q_j$ for all $1\leq j\leq n$ (in their simplest form) are both coprime to $\pi,\pi'$. Our final simplification is that we allow no error term in \eqref{eq:def-of-S}, which we will denote by writing $\infty$ in place of $M$.

Our convention will be that vectors, i.e. elements of $\ZZ[i]^n,\CC^n$, etc. are always meant as column vectors. We also introduce the notation $v_{\pi}(q)$, for any prime $\pi\in\ZZ[i]$ and any $q\in\QQ(i)$, which denotes the $\pi$-valuation of $q$. Occasionally, we may use this notation for vectors or matrices with entries from $\QQ(i)$, then it means the minimal $\pi$-adic valuation attained by the entries.

\begin{lemma}\label{lemma:polarization}
Let $\pi\in\ZZ[i]$ be a prime lying above a split rational prime $p=\pi\overline{\pi}$, $\rho\in\NN$, and let $A\in\QQ(i)^{n\times n}$ be a self-adjoint matrix such that $v_{\pi}(A)\geq 0$. Let $x=(\xi_1,\dotsc,\xi_n)^t,y=(\upsilon_1,\dotsc,\upsilon_n)^t\in\ZZ[i]^n$ be vectors satisfying $v_{\pi}(x)=v_{\pi}(y)=0$ such that for any $1\leq j<k\leq n$, $\pi^{\rho}\mid (\xi_j\upsilon_k-\xi_k\upsilon_j)$. Then the following statements hold. 
\begin{enumerate}[(a)]
\item\label{polarization-part-a} For some $a\in\ZZ$ coprime to $p$, we have $y\equiv ax \bmod \pi^{\rho}$, and this $a$ is well-defined modulo $p^{\rho}$.
\item\label{polarization-part-b} Choosing $a\in\ZZ$ with this property in \eqref{polarization-part-a}, there also exists a $b\in\ZZ$ (unique modulo $p^\rho$) with $b\equiv ai\bmod \pi^\rho$, so we further have
\begin{equation}\label{eq:polarization}
2 x^* A y \equiv (a-bi) x^* A x + (a'-b'i)  y^* A y \qquad \bmod p^{\rho} ,
\end{equation}
where $a'\in \ZZ$ (resp.\ $b'\in \ZZ$) stands for the multiplicative inverse of $a$ (resp.\ of $b$) modulo $p^{\rho}$.
\item\label{polarization-part-c} We have $\pi^\rho\mid a'-b'i$ and $\overline{\pi}^\rho\mid a-bi$ for the integers $a,b,a',b'$ defined above.
\end{enumerate}
\end{lemma}
\begin{proof}
This is a variant of \cite[Lemma~3]{BlomerMaga2016}, but the proof in this case is a little more computational. Without loss of generality, we may assume that $\pi\nmid \xi_n$. Then $\pi\nmid \upsilon_n$, for if not, then for some $1\leq j\leq n-1$, $\pi\nmid \upsilon_j$, and then $\pi\nmid (\xi_n\upsilon_j-\xi_j\upsilon_n)$. Then $a\equiv \upsilon_n \xi_n^{-1} \bmod \pi^{\rho}$ does the job, since $\upsilon_j\equiv \xi_j \upsilon_n \xi_n^{-1} \equiv a\xi_j \bmod \pi^{\rho}$ for any $1\leq j\leq n-1$. Also, $a$ can be chosen in $\ZZ$ uniquely modulo $p^\rho$ since $0,1,\dots,p^{\rho}-1$ is a set of representatives modulo $\pi^{\rho}$ (i.e. we have the isomorphism $\ZZ/(p^\rho)\cong \ZZ[i]/(\pi^\rho)$). The proof of \eqref{polarization-part-a} is complete.

As for \eqref{polarization-part-b}, fix any representative $a\in\ZZ$ and let $b\in \ZZ$ be such that $b\equiv ai\bmod \pi^\rho$. Then $\pi^\rho$ divides all the entries in $y-ax$ whence $\overline{\pi}^\rho$ divides all the entries in $(y-ax)^*$. So we deduce
\[
0\equiv (y-ax)^*A(y-ax)=y^*Ay+a^2x^*Ax-ax^*Ay-ay^*Ax \qquad\bmod p^\rho=(\pi\overline{\pi})^\rho
\]
hence
\begin{equation}\label{real}
x^*Ay+y^*Ax\equiv ax^*Ax + a' y^*Ay\qquad\bmod p^\rho
\end{equation}
where $a'\in \ZZ$ is a multiplicative inverse of $a$ modulo $p^\rho$. Similarly, $\pi^\rho$ divides all the entries in $iy-bx$ and $\overline{\pi}^\rho$ divides all the entries in $(iy-bx)^*$ as $ai\equiv b\bmod \pi^\rho$. So we compute
\[
0\equiv (iy-bx)^*A(iy-bx)=y^*Ay+b^2x^*Ax-bix^*Ay+biy^*Ax \qquad\bmod p^\rho
\]
whence
\begin{equation}\label{imaginary}
x^*Ay-y^*Ax\equiv -bix^*Ax -b'i y^*Ay \qquad\bmod p^\rho
\end{equation}
where $b'\in \ZZ$ is a multiplicative inverse of $b$ modulo $p^\rho$. The statement follows by adding equations \eqref{real} and \eqref{imaginary}.

For part \eqref{polarization-part-c}, we compute $\pi^\rho\mid i(b-ai)=ib+a$ whence $\overline{\pi}^\rho\mid \overline{a+bi}=a-bi$. Similarly, $\pi^\rho\mid a'b'(b-ai)\equiv a'-b'i$.
\end{proof}

\begin{lemma}\label{lemma:basic-count} Let $A\in\Omega'$, and assume that $x_1,\dots,x_k\in\ZZ[i]^n$ are linearly independent vectors for some $0\leq k\leq n-1$. Then for any real number $\beta \geq 2$,
\[
\#\{y\in \ZZ[i]^n: y^* A y = \beta^2, \text{ and } x_j^* A y=0 \text{ for any } j=1,\dotsc,k\}\ll_{\Omega,\eps} \beta^{2(n-k-1)+\eps}
\]
for any $\eps>0$.
\end{lemma}
\begin{proof}
See \cite[Corollary~5.3]{BlomerMaga2015}.
\end{proof}

We slightly extend the notation $S(Q,\pi^{\nu},\pi'^{\nu},\infty)$ for $\pi,\pi'\nmid m\in\ZZ[i]$ as
\[
\begin{split}
S_m(Q,\pi^{\nu},\pi'^{\nu},\infty): =  \Bigl\{ \gamma\in \SL_n(\ZZ[i]_{\pi,\pi'})&\cdot\diag(1,\pi^{\nu},\dotsc,\pi^{\nu},\pi^{\nu}\pi'^{\nu})\cdot \SL_n(\ZZ[i]_{\pi,\pi'}): \\ & m\gamma\in\SL_n(\ZZ[i]), |\det\gamma|^{-\frac{2}{n}}\cdot \gamma^* Q \gamma =  Q  \Bigr\},
\end{split}
\]
where by $\ZZ[i]_{\pi,\pi'}$, we mean the ring of elements $a\in \QQ(i)$ which satisfy $v_{\pi}(a),v_{\pi'}(a)\geq 0$.

\begin{lemma}\label{lemma:spec-lemma-one-prime} Let $\pi\in\ZZ[i]$ be a prime lying above a split rational prime, and $\pi\nmid m\in\ZZ[i]$. Let $Q=\diag(q_1,\dotsc,q_n)$ with $q_j\in\QQ$, $q_j\asymp_{\Omega} 1$ for all $1\leq j\leq n$. Then for any $1\leq \nu\leq n$,
\[
\#S_m(Q,\pi^{\nu},\pi^{\nu},\infty)\ll_{\Omega,\eps} |m|^{2n^2-2+\eps} \den(Q)^{\frac{(2n-1)(n-1)}{2}} N(\pi)^{\nu(n-1)+\eps}.
\]
for any $\eps>0$. Here, $\den(Q)$ denotes the least common multiple of the denominator of the diagonal entries of $Q$ (written in simplest form).
\end{lemma}
\begin{proof}
By the definition of $S_m(Q,\pi^{\nu},\pi^{\nu},\infty)$, any matrix counted there must have a column not completely divisible by $\pi$, and we may assume that it is the first column $\gamma_1$. Since $(m\gamma_1)^* Q (m\gamma_1) = |m|^2N(\pi)^{\nu}$, we have, by Lemma~\ref{lemma:basic-count} that the number of possible $m\gamma_1$'s is $O_{\eps}(|m|^{2(n-1)+\eps} N(\pi)^{\nu(n-1)+\eps})$. Now it suffices to prove that fixing $\gamma_1$, there are $O_{\Omega}(m^{2n(n-1)}\den(Q)^{(2n-1)(n-1)/2})$ ways to finish the matrix.

We group the possible second columns $\gamma_2$ according to their $\pi$-valuation. When $v_{\pi}(\gamma_2)\geq \nu$, then $m\gamma_2/\pi^{\nu}\in\ZZ[i]^{n}$ and $\|m\gamma_2/\pi^{\nu}\|\asymp_{\Omega} \|m\gamma_2/\pi^{\nu}\|_Q=|m|$, hence there are $O_{\Omega}(|m|^{2n})$ such choices for such a $\gamma_2$. Now fix $0\leq \mu < \nu$, and then it suffices prove that the number of possible $\gamma_2$'s satisfying $v_{\pi}(\gamma_2)=\mu$ is $O_{\Omega}(|m|^{2n}\den(Q)^{(2n-1)/2})$, since the same can be repeated for all further columns.

Let $x,y$ be second columns after the first column $\gamma_1$ such that $v_{\pi}(x)=v_{\pi}(y)=\mu$. Consider then the vectors $x':=mx/\pi^{\mu}$ and $y':=my/\pi^{\mu}$. By Lemma~\ref{lemma:polarization}\eqref{polarization-part-a}, $x'$ and $y'$ are both multiples of $m\gamma_1$ modulo $\pi^{\nu-\mu}$, and then by transitivity, multiples of each other modulo $\pi^{\mu-\nu}$. Introduce then $Q':=\den(Q)Q$. With this notation, since $x'^*Q' x'=y'^* Q' y'=|m|^2 q_2\den(Q)\pi^{\nu-\mu}\overline{\pi}^{\nu-\mu}=|m|^2 q_2\den(Q)N(\pi)^{\nu-\mu}$, \eqref{eq:polarization} gives that $x'^* Q' y'$ is divisible by $N(\pi)^{\nu-\mu}$. On the other hand, viewing $\CC^n$ as $\RR^{2n}$, we obtain, for the $Q$-angle of $x$ and $y$ that
\[
\sphericalangle_{Q}(x,y) = \arccos \frac{\Re (x^* Q y)}{\sqrt{(x^* Q x) (y^* Q y)}} = \arccos \frac{\Re (x'^* Q' y')}{\sqrt{(x'^* Q' x') (y'^* Q' y')}} = \arccos\frac{\Re (x'^* Q' y')}{|m|^2 q_2\den(Q) N(\pi)^{\mu-\nu}}.
\]
Since the numerator in the rightmost expression is divisible by $N(\pi^{\mu-\nu})$, this continues, for some $\ell\in\{0,1,2,\dotsc\}$, as
\[
\sphericalangle_{Q}(x,y) = \arccos\left(1-\frac{\ell}{|m|^2 q_2\den(Q)}\right) \gg_{\Omega} |m|^{-1}\den(Q)^{-1/2}, \qquad \text{if $\ell\neq 0$, i.e. $x\neq y$}.
\]
This means that the number of possible second columns is $O_{\Omega}(|m|^{2n-1} \den(Q)^{(2n-1)/2})$ (this is rather elementary, see also \cite[Lemma~4]{BlomerMaga2016}). The proof is complete.
\end{proof}

\begin{lemma}\label{lemma:spec-lemma-two-primes} Let $\pi,\pi'\in\ZZ[i]$ be distinct primes lying above distinct split rational primes, and let $m\in\ZZ[i]$ such that $v_{\pi}(m)=v_{\pi'}(m)=0$. Let further $Q=\diag(q_1,\dotsc,q_n)$ such that $q_j\in\QQ$ and $v_{\pi}(q_j)=v_{\pi'}(q_j)=0$ for all $1\leq j\leq n$. Then for any $1\leq \nu\leq n$,
\[
\#S_m(Q,\pi^{\nu},\pi'^{\nu},\infty) = 0.
\]
\end{lemma}
\begin{proof} First observe that if $1\leq \nu \leq n-1$, then
\[
|\det \gamma|^{\frac{2}{n}}=\left(\pi\overline{\pi}\right)^{\frac{\nu(n-1)}{n}}\left(\pi'\overline{\pi'}\right)^{\frac{\nu}{n}}\notin\QQ,
\]
hence $|\det \gamma|^{2/n} \gamma^* Q \gamma \neq Q$, implying $\#S_m(Q,\pi^{\nu},\pi'^{\nu},\infty)=0$ in this case, so from now on, we assume $\nu=n$.

Any $\gamma$ counted in $\#S_m(Q,\pi^n,\pi'^n,\infty)$ must have a column, say, the first one $\gamma_1$ such that $v_{\pi}(\gamma_1)=0$. Now we prove that any choice of $\gamma_2$ leads to a contradiction. Let $\mu:=v_{\pi}(\gamma_2)\geq 0$. Then we apply Lemma~\ref{lemma:polarization}, in particular, \eqref{eq:polarization} for $\gamma_1$ and $\gamma_2':=\gamma_2/\pi^\mu$, with the notation $p=\pi\overline{\pi}$, $a,b\in\ZZ$ coprime to $p$ such that $1\leq a,b,a',b'\leq p^{n-\mu}-1$, $\gamma_2'\equiv a\gamma_1 \bmod \pi^{n-\mu}$, $b\equiv ai \bmod \pi^{n-\mu}$, $a',b'$ are the multiplicative inverse of $a,b$ modulo $p^{n-\mu}$, respectively. We infer
\begin{equation}\label{eq:distinct-primes}
0=\pi^{-\mu}\gamma_1^* Q \gamma_2 =\gamma_1^* Q \gamma'_2 \equiv (a-bi)p^{n-1} \pi'\overline{\pi'} q_1 + (a'-b'i)p^{n-1-\mu} \pi'\overline{\pi'} q_2 \qquad \bmod p^{n-\mu},
\end{equation}
which is immediately a contradiction for $\mu\geq 1$, since the second term in the rightmost expression is not divisible by $p^{n-\mu}$, while the first one and the sum are.

Now assume $\mu=0$. By Lemma \ref{lemma:polarization}\eqref{polarization-part-c} we have $\overline{\pi}\mid a-bi$ and $\pi\mid a'-b'i$. So in \eqref{eq:distinct-primes} the first term on the right-hand side is not divisible by $\pi^n$ (since $\pi\nmid a-bi$ as $p\nmid a-bi$), while the second one and the sum are.
\end{proof}

\section{Exchanging matrices}\label{sec:exchanging}

Recall that we want to prove the countings \eqref{eq:prop-counting-one-prime} and \eqref{eq:prop-counting-two-primes} for $N(\pi),N(\pi')\in[L,2L)$ with an appropriate choice of $L\geq L_0$ not exceeding a fixed power of $L_0$. Our strategy is to switch $Q$ for other matrices with better and better arithmetic properties. Note that a priori, $Q$ is a point on a real manifold with entries that might be highly transcendental. We first informally describe these switches. First we will write $Q_1$ in place of $Q$ in order to guarantee that
\[
S(Q,\pi^{\nu},\pi'^{\nu},M)\subseteq S(Q_1,\pi^{\nu},\pi'^{\nu},\infty)
\]
for every admissible choice of $\pi,\pi',\nu$ satisfying that $N(\pi),N(\pi')$ are between two fixed powers of $L_0$. This new $Q_1$ will have entries in a number field $K\supseteq \QQ(i)$. Using this $Q_1$, we will be able to show that in an appropriate subinterval (again, the norms are between two fixed powers of $L_0$), all the $\gamma$'s correspond to $\pi=\pi'$ or $\nu=n$.

Secondly, in this subinterval,  we will find a $Q_2$, this time satisfying that
\[
S(Q,\pi^{\nu},\pi'^{\nu},M) \subseteq S(Q_2,\pi^{\nu},\pi'^{\nu},\infty)
\]
for every admissible choice of $\pi,\pi',\nu$ satisfying that $N(\pi),N(\pi')$ fall into an even shorter subinterval (again, between two powers of $L_0$). This new $Q_2$ will have entries in $\QQ(i)$, and we will have a control on their height. Finally, we will diagonalize $Q_2$ into $Q_3$, again, with a control on the height of the entries, which necessarily will be rational. This diagonalization process will affect the counted $\gamma$'s themselves, but not their number, and we will have a good control on the possible denominators of the newly counted $\gamma$'s, i.e.
\[
\#S(Q,\pi^{\nu},\pi'^{\nu},M)\leq \#S(Q_2,\pi^{\nu},\pi'^{\nu},\infty) \leq \#S_m(Q_3,\pi^{\nu},\pi'^{\nu},\infty)
\]
for $\pi,\pi',\nu$ as above. Then to the rightmost count here, we will apply Lemmata~\ref{lemma:spec-lemma-one-prime}--\ref{lemma:spec-lemma-two-primes}, which in fact will be of the quality of \eqref{eq:prop-counting-one-prime} and \eqref{eq:prop-counting-two-primes}.

Now we carry out this plan in detail. First we formulate a statement on effective computability. The field of algebraic numbers is denoted by $\overline{\QQ}$. Given an algebraic number $a$, we define the complexity of $a$ as
\[
\comp(a):= \inf_{\substack{a=\frac{b}{c} \\ b,c\in\mathcal{O}_{\QQ(a)}}} \left( \max_{\sigma\in\Hom_{\QQ}(\QQ(a),\CC)} |\sigma(b)| + \max_{\sigma'\in\Hom_{\QQ}(\QQ(a),\CC)} |\sigma'(c)| \right) + 1
\]
with $\mathcal{O}_{\QQ(a)}$ standing for the ring of integers of $\QQ(a)$.
\begin{lemma}\label{lemma:effective-computability} Let $a_1,\dotsc,a_m$ be algebraic numbers and put $K:=\QQ(a_1,\dotsc,a_m)$. Assume that $f:\overline{\QQ}^m\to\overline{\QQ}$ is a function of $m$ variables that is computed altogether by $t$ additions, subtractions, multiplications and divisions. Then we have
\[
\comp(f(a_1,\dotsc,a_m))\leq \left(\max_{1\leq j\leq m}(\comp(a_j))\right)^{O_{m,t,\deg(K:\QQ)}(1)}\ .
\]
\end{lemma}
\begin{proof}
See the first paragraph of the proof of \cite[Lemma~5]{BlomerMaga2016}.
\end{proof}

Include $\Omega'$ in some $\Omega_1$ which is still a compact subset of $P_n$ in such a way that $\Omega'\subset\opint(\Omega_1)$. We make this inclusion in a well-defined way, say, $\Omega_1$ is the set of self-adjoint, positive definite matrices of eigenvalues between $a/2$ and $2b$, where $a,b$ are respectively the smallest and the largest eigenvalues of matrices in $\Omega'$. Then $\dist(\Omega',P_n\setminus\Omega_1)\gg_{\Omega} 1$, where by the distance of matrices, we mean the one induced by the maximum of the entrywise distance.

Given any matrix $\gamma\in\GL_n(\CC)$, we define the following linear transformation $B_{\gamma}:S_n\to S_n$:
\[
B_{\gamma}(A):=\gamma^* A \gamma - |\det(\gamma)|^{\frac{2}{n}} A,\qquad A\in S_n.
\]

For any $0< C_1<C_2$, we introduce the notation $\cP(C_1,C_2)$ for the set of primes $\pi\in\ZZ[i]$ satisfying that $\Re \pi,\Im\pi >0$ and $C_1\leq N(\pi)< C_2$ (then in particular, elements of $\cP(C_1,C_2)$ lie above distinct split rational primes). Another notation we introduce is $T\geq 1$, a fixed number depending only on $n$, which exceeds all the implicit constants in $O(1)$'s ever mentioned in the paper.

\begin{lemma}\label{lemma:constructQ1} Let $D\geq 2$ and $E\geq 2$ be arbitrary. Then there exist some $M\geq 2$ (depending on $n,D,E$) and $L_1(\Omega)\geq 2$ (depending on $n,D,E,M,\Omega$) such that the following holds. For any $Q\in\Omega'$, there exists some $1\leq j\leq n^2+1$ satisfying that
\[
S(Q,\pi^{\nu},\pi'^{\nu},M)=\emptyset
\]
for all $\pi,\pi'\in\cP(2L_0^{(DE)^j},2L_0^{(DE)^{j+1}})$, unless $\pi=\pi'$ or $\nu=n$.
\end{lemma}
\begin{proof} Fix $Q\in \Omega'$. For any $L_0\geq 2$, consider the subspaces (for the moment, with some unspecified $M\geq 2$)
\[
H_j:= \bigcap_{\substack{\gamma\in S(Q,\pi^{\nu},\pi'^{\nu},M) \\ \pi,\pi'\in \cP(L_0,2L_0^{(DE)^j}) \\ 1\leq \nu\leq n}} \ker B_{\gamma},\qquad j=1,\dotsc,n^2+2.
\]
Then $S_n\supseteq H_1\supseteq \dotsc \supseteq H_{n^2+2}\supseteq \{0\}$ and $\dim(S_n)=n^2$ imply that $H_j=H_{j+1}$ for some $1\leq j\leq n^2+1$. Fix the smallest such $j$. Since $\dim S_n-\dim H_j\leq n^2$, we in fact obtain $H_j$ by intersecting only $n^2$ many $\ker B_{\gamma}$'s, say,
$H_j=\cap_{\ell=1}^{t} \ker B_{\gamma_{\ell}}$
with some $t\leq n^2$. Since each entry of each such $\gamma_{\ell}$ has complexity $O_{\Omega}((L_0^{(DE)^j})^{O(1)})$, $H_j$ is defined via a system of linear equations with entries from $K=\QQ(i,\pi_1^{1/n},\dotsc,\pi_{2n^2}^{1/n})$ of complexity $O_{\Omega}((L_0^{(DE)^j})^{O(1)})$ by Lemma~\ref{lemma:effective-computability} (since the linear system defining $H_j$ can be computed in $O(1)$ steps from the used $\gamma$'s).

By assumption, $\dist(Q,\ker B_{\gamma})\ll L_0^{-M}$, where by $\dist$, we mean the distance in $S_n$ as a real vector space of dimension $n^2$. We claim that this implies that $\dist(Q,H_j)=O_{\Omega}((L_0^{(DE)^j})^{O(1)}L_0^{-M})$. Indeed, take basis matrices $V_1,\dotsc,V_m$ of the linear span of the $(\ker B_{\gamma_{\ell}})^{\perp}$'s such that each $V_{\ell'}$ is in one of the $\ker B_{\gamma_{\ell}}$'s, its entries are in $K$ of complexity $O_{\Omega}((L_0^{(DE)^j})^{O(1)})$ (this can be done, because such a basis can be computed from the $\gamma_{\ell}$'s, and then Lemma~\ref{lemma:effective-computability} applies). Such a basis can be orthogonalized by Gram--Schmidt into $V'_1,\dotsc,V'_m$, and then
\[
\dist(Q,H_j)=\|\proj_{H_j^{\perp}}(Q)\|= \left \|\sum_{\ell'=1}^m \frac{\langle Q,V'_{\ell'} \rangle}{ \langle V'_{\ell'} ,V'_{\ell'} \rangle}  V'_{\ell'} \right\| = \left \| \sum_{\ell'=1}^m \frac{ \sum_{\ell''=1}^m U_{\ell'\ell''}\langle Q,V_{\ell''} \rangle}{ \langle V'_{\ell'} ,V'_{\ell'} \rangle}  V'_{\ell'} \right\|,
\]
where $U$ is the matrix standing for the Gram--Schmidt process. Since $\langle Q,V_{\ell''} \rangle \ll L_0^{-M}$ (because for every $\ell''$, we have $V_{\ell''}\in \ker B_{\gamma_{\ell}}$ for some $\ell$), and all the coefficients (including the entries of $U$, by Lemma~\ref{lemma:effective-computability}) are of complexity $O_{\Omega}((L_0^{(DE)^j})^{O(1)})$, we indeed see that $\dist(Q,H_j)=O_{\Omega}((L_0^{(DE)^j})^{O(1)}L_0^{-M})$, applying Lemma~\ref{lemma:effective-computability} again.

Now if $M$ is large enough, say,
\begin{equation}\label{eq:condition-on-M}
M\geq T\cdot (DE)^{n^2+2} + 1,
\end{equation}
where recall that $T\geq 1$ is an upper bound on all the implicit constants in $O(1)$'s ever mentioned in the paper, then $\dist(Q,H_j)=O_{\Omega}(L_0^{-1})$ (since $j\leq n^2+1$, the exponent $n^2+2$ seems to be an overkill, but for a later reference, it is better to force $M$ even a slightly larger). Then with the convenient choice of $L_1(\Omega)\geq 2$, if $L_0\geq L_1(\Omega)$, then $H_j\cap \opint\Omega_1\neq\emptyset$, in particular, $H_j\neq \{0\}$. Further, since $H_j$ is defined over $K$, we may find and fix some $0\neq Q_1 \in H_j$ with entries in $K$.

By the definition of $H_j$ and its choice $H_j=H_{j+1}$, we have that for any $\pi,\pi'\in \cP(L_0,2L_0^{(DE)^{j+1}})$ and any $1\leq \nu\leq n$,
\[
S(Q,\pi^{\nu},\pi'^{\nu},M) \subseteq S(Q_1,\pi^{\nu},\pi'^{\nu},\infty).
\]

Assume that $\gamma\in S(Q,\pi^{\nu},\pi'^{\nu},M)$ for some $\pi,\pi'\in\cP(2L_0^{(DE)^j},2L_0^{(DE)^{j+1}})$. Since $Q_1$ has a nonzero entry, say, $Q_{rs}$, this implies via the last display that
\[
\gamma_r^* Q_1 \gamma_s = |\det(\gamma)|^{\frac{2}{n}} Q_{rs},
\]
where $\gamma_r,\gamma_s$ stand for the $r$th and $s$th column of $\gamma$, respectively. Here, the left-hand side is in $K$, so is the right-hand side, which implies that $|\det(\gamma)|^{2/n}=(\pi\overline{\pi})^{\nu(n-1)/n}(\pi'\overline{\pi'})^{\nu/n}\in K$. By the independence of roots (a theorem of Besicovitch \cite{Besicovitch}) and that $K$ is defined by $n$th roots of primes of norm less than $2L^{(DE)^j}$, while primes in the definition of $H_{j+1}$ have norms at least $2L^{(DE)^j}$, we see that $\pi=\pi'$ or $\nu=n$, and the proof is complete.
\end{proof}

\begin{lemma}\label{lemma:constructQ2}
Let $D\geq 2$ be arbitrary, and assume that $E>D^{n^2+1}$. Let $M,L_1(\Omega)\geq 2$ be given by Lemma~\ref{lemma:constructQ1} and \eqref{eq:condition-on-M}. For any $Q\in \Omega'$, and for the corresponding $j\leq n^2+1$ given by Lemma~\ref{lemma:constructQ1}, there exists a self-adjoint matrix $Q_2$ with entries in $\mathbb{Q}(i)$ of complexity $O_{\Omega}((L_0^{D^k(DE)^j})^{O(1)})$ for some $k\in \{0,\dots,n^2\}$ such that for any $\pi,\pi' \in \cP(2 L_0^{D^k(DE)^j},2 L_0^{D^{k+1}(DE)^j})$ and any $1\leq \nu\leq n$,
\[
S(Q,\pi^\nu,\pi'^\nu,M)
\subseteq 
S(Q_2,\pi^\nu,\pi'^\nu,\infty).
\]
\end{lemma}
\begin{proof}
Define the subspaces
\[
H'_k:= \bigcap_{\substack{\gamma\in S(Q,\pi^{\nu},\pi'^{\nu},M) \\ \pi,\pi'\in \cP(2 L_0^{D^k(DE)^j},2 L_0^{D^{k+1}(DE)^j}) \\ 1\leq \nu\leq n}} \ker B_{\gamma},\qquad k=0,\dotsc,n^2+1.
\]
Then $S_n\supseteq H'_0\supseteq \dotsc \supseteq H'_{n^2+1}\supseteq \{0\}$ and $\dim(S_n)=n^2$ imply that $H'_k=H'_{k+1}$ for some $1\leq k\leq n^2$. Fix the smallest such $k$. 

By our choice $D^{n^2+1}<E$ we have $2L_0^{(DE)^j}\leq 2 L_0^{D^k(DE)^j}<2 L_0^{D^{k+1}(DE)^j}< 2L_0^{(DE)^{j+1}}$, so we may apply Lemma \ref{lemma:constructQ1} to deduce $S(Q,\pi^\nu,\pi'^\nu,M)=\emptyset$ unless $\pi=\pi'$ or $\nu=n$. In particular, we have $|\det(\gamma)|^{2/n}=(\pi\overline{\pi})^{\nu(n-1)/n}(\pi'\overline{\pi'})^{\nu/n}\in \ZZ$ for all $\gamma\in S(Q,\pi^{\nu},\pi'^{\nu},M)$ in the definition of $H'_k$. Therefore all linear maps $B_\gamma$ in the definition of $H'_k$ are defined over $\QQ(i)$. As in the proof of Lemma \ref{lemma:constructQ1}, we find a matrix $Q_2\in \opint\Omega_1 \cap H'_k$ with entries in $\QQ(i)$ of complexity $O_{\Omega}((L_0^{D^k(DE)^j})^{O(1)})$ by Lemma~\ref{lemma:effective-computability} such that the conclusion holds. (Note that this is the point where we use the slightly stronger condition \eqref{eq:condition-on-M} put on $M$ that it exceeds not only $(DE)^j\cdot O(1)$, but $D^k(DE)^j\cdot O(1)$.)
\end{proof}

\begin{lemma}\label{lemma:constructQ3}
Let $D\geq 2$ be arbitrary, and let $E$ be as in Lemma~\ref{lemma:constructQ2}, then $M,L_1(\Omega)$ as given by Lemma~\ref{lemma:constructQ1}. For any $Q\in\Omega_1$, let $j,k$ be the numbers given also by Lemmata~\ref{lemma:constructQ1}--\ref{lemma:constructQ2}. There exists some $m\in \ZZ[i]$ of complexity $O_{\Omega}((L_0^{D^k(DE)^j})^{O(1)})$ and a matrix $U\in \SL_n(\ZZ[i,\frac{1}{m}])$ with the following properties:
\begin{enumerate}[$(a)$]
\item\label{Q3-part-a} All entries of $U$ and $U^{-1}$ have complexity $O_{\Omega}((L_0^{D^k(DE)^j})^{O(1)})$.
\item\label{Q3-part-b} $Q_3:=U^* Q_2 U$ is diagonal with entries in $\QQ$ of complexity $O_{\Omega}((L_0^{D^k(DE)^j})^{O(1)})$ and lies in a compact subset $\Omega_2\subset P_n$ depending only on $\Omega$.
\item\label{Q3-part-c} For any 
$\pi,\pi'\in \cP(2 L_0^{D^k(DE)^j},2 L_0^{D^{k+1}(DE)^j})$ not dividing $m$, and any $1\leq \nu\leq n$, we have
\[
\#S(Q_2,\pi^{\nu},\pi'^{\nu},\infty) \leq \# S_m(Q_3,\pi^{\nu},\pi'^{\nu},\infty).
\]
\end{enumerate}
\end{lemma}

\begin{proof}
We construct $U$ in a way such that $Q_3$ is a Gram--Schmidt orthogonalized form of $Q_2$, i.e.
\[
Q_3=U^* Q_2 U,
\]
where $U$ is an $\binom{n}{2}$ long composition of elementary base changes $\id_n+uE_{\ell_1\ell_2}$ for $\ell_1\neq \ell_2$, where $u$ is chosen to eliminate the entry at position $\ell_1,\ell_2$ (note that $U$ is not canonically determined, since we can eliminate the elements in any order, fix one once for all). Noting that each entry of $Q_2$ has complexity $O_{\Omega}((L_0^{D^k(DE)^j})^{O(1)})$, this proves \eqref{Q3-part-a} by Lemma~\ref{lemma:effective-computability}.

As for \eqref{Q3-part-b}, first we claim that in a Gram--Schmidt step applied to a positive definite matrix, one diagonal entry decreases in absolute value, and all the others remain the same. It suffices to check this for $2\times 2$ blocks, where we see the calculation
\[
\begin{pmatrix}
1 & 0 \\ -\frac{\overline{b}}{a} & 1
\end{pmatrix}
\begin{pmatrix}
a & b \\ -\overline{b} & c
\end{pmatrix}
\begin{pmatrix}
1 & -\frac{b}{a} \\ 0 & 1
\end{pmatrix}
=
\begin{pmatrix}
a & 0 \\ 0 & c - \frac{|b|^2}{a}
\end{pmatrix}.
\]
Here, $0<c-|b|^2/a<c$, so the claim is proven.

As a result, when we diagonalize any element of $\Omega_1$, the resulting diagonal matrix cannot have entry larger than the largest eigenvalue in $\Omega_1$, say, $\lambda$. But then the smallest eigenvalue cannot be smaller than $\Delta/\lambda^{n-1}$, where $\Delta$ is the smallest determinant attained in $\Omega_1$. Therefore, any Gram--Schmidt process applied to any element of $\Omega_1$ leads to positive definite matrices with eigevalues between $\Delta/\lambda^{n-1}$ and $\lambda$. The set of such matrices is a good choice for $\Omega_2$.

The rationality of $Q_3$ is obvious, since it is self-adjoint and with diagonal entries a priori in $\QQ(i)$. Then observe that the matrix $Q_3$ is computed in $O(1)$ many steps from $Q_2$, which verifies via Lemma~\ref{lemma:effective-computability} that the entries of $Q_3$ are indeed of complexity $O_{\Omega}((L_0^{D^k(DE)^j})^{O(1)})$. This proves \eqref{Q3-part-b}.

Also, put $m:=\den(U^{-1})\den(U)$ (which is $\den(U)^2$, since $\det(U)=1$). Then $m$ can be computed in $O(1)$ steps, and referring to Lemma~\ref{lemma:effective-computability}, we see that the complexity of $m$ is $O_{\Omega}((L_0^{D^k(DE)^j})^{O(1)})$ as claimed. In particular, we have $mU^{-1}\gamma U\in \ZZ[i]^{n\times n}$ for any $\gamma\in \SL_n(\ZZ[i])$.

Finally, if $\pi,\pi'\nmid m$, then the base change given by $U$ does not alter the $\pi$- and $\pi'$-parts of the Smith normal form. Hence if $\gamma\in S(Q_2,\pi^{\nu},\pi'^{\nu},\infty)$, then
\[
|\det(\gamma)|^{\frac{2}{n}}Q_3 = |\det(\gamma)|^{\frac{2}{n}} U^* Q_2 U = U^* \gamma^* Q_2 \gamma U  = (U^{-1}\gamma U)^* U^* Q_2 U (U^{-1}\gamma U) = (U^{-1}\gamma U)^* Q_3 (U^{-1}\gamma U)
\]
shows that $U^{-1}\gamma U\in S_m(Q_3,\pi^{\nu},\pi'^{\nu},\infty)$. Since the $U$-conjugation is a bijection, we obtain \eqref{Q3-part-c}.
\end{proof}

\section{The endgame}

Let $\eps>0$ be given as in the input of Proposition~\ref{prop:main}. Choose then $D\geq 2$ such that $D \eps/2>T+1$, where recall that $T\geq 1$ is an upper bound on all implicit constants in $O(1)$'s ever mentioned in the paper. Let then $E\geq 2$ as needed in Lemma~\ref{lemma:constructQ2}, and then $M\geq 2$ as implied by Lemma~\ref{lemma:constructQ1}. This $M$ will be the $M(\eps)$ of Proposition~\ref{prop:main}.

Now let $\Omega$, hence $\Omega',\Omega_1,\Omega_2$ be given. Let $L(\Omega)$ be large enough such that on the other hand $L(\Omega) \geq L_1(\Omega)$ of Lemmata~\ref{lemma:constructQ1}--\ref{lemma:constructQ2}, and on the other hand that for any $L_0\geq L(\Omega)$, and any $m$ implied by Lemma~\ref{lemma:constructQ3} (for any possible $1\leq j\leq n^2+1$ and $0\leq k\leq n^2$), $N(m)<L_0^{D^{k+1}(DE)^j}$, which can be achieved, since $N(m)=O_{\Omega}(L_0^{D^k(DE)^j})$. This in particular implies $\pi\nmid m$ for any $\pi\in \cP(L_0^{D^{k+1}(DE)^j},2L_0^{D^{k+1}(DE)^j})$. Note that at this point none of $j,k,m$ is fixed, but we have the claimed bounds and the non-divisibility relations on them.

Let then $Q\in\Omega'$ and $L_0\geq L(\Omega)$ be arbitrary. Then there exist some $1\leq j\leq n^2+1$ given by Lemma~\ref{lemma:constructQ1}, $0\leq k\leq n^2$ and $Q_2$ given by Lemma~\ref{lemma:constructQ2}, $Q_3$ and $m\in\ZZ[i]$ given by Lemma~\ref{lemma:constructQ3} with the properties given there. Let then $L:=L_0^{D^{k+1}(DE)^j}$, which satisfies the magnitude requirement of Proposition~\ref{prop:main} that $L_0\leq L\leq L_0^{\alpha(\eps)}$, where $\alpha(\eps)$ is a constant depending only on $\eps$ (we can take $\alpha(\eps)=D^{n^2+1}(DE)^{n^2+1}$). For any $\pi,\pi'\in\cP(L,2L)$ and any $1\leq \nu \leq n$, we have, combining Lemmata~\ref{lemma:constructQ2}--\ref{lemma:constructQ3}, that
\[
\#S(Q,\pi^{\nu},\pi'^{\nu},M) \leq \#S(Q_2,\pi^{\nu},\pi'^{\nu},\infty) \leq \#S_m(Q_3,\pi^{\nu},\pi'^{\nu},\infty).
\]
Therefore, it suffices to estimate the rightmost expression from above. We apply Lemma~\ref{lemma:spec-lemma-two-primes} for $\pi\neq \pi'$ to see the count is $0$, which proves \eqref{eq:prop-counting-two-primes}. When $\pi=\pi'$, then we apply Lemma~\ref{lemma:spec-lemma-one-prime} (with $\eps/2$ written in place of $\eps$ there) to see it is
\[
O_{\Omega}((L_0^{D^k(DE)^j})^{O(1)}) \cdot O_{\Omega,\eps} ((L_0^{D^{k+1}(DE)^j})^{\nu(n-1)+\frac{\eps}{2}}) = O_{\Omega,\eps}(L^{\nu(n-1)+\eps}),
\]
where the last bound holds by the choice of $D$ and by $N(\pi)< 2L$. This proves \eqref{eq:prop-counting-one-prime}.

The proof of Proposition~\ref{prop:main} and hence that of Theorem~\ref{thm:main} are complete.

\bibliographystyle{alpha}
\bibliography{main}

\end{document}